\documentclass[reqno,12pt]{amsart}

\usepackage[utf8]{inputenc}

\usepackage{enumerate}
\usepackage[margin=1in]{geometry}
\usepackage{ifpdf}
\usepackage{amsmath}
\usepackage{amsfonts}
\usepackage{amssymb}
\usepackage{amsthm}
\usepackage[ocgcolorlinks,hyperfootnotes=false,colorlinks=true,citecolor=blue,linkcolor=blue,urlcolor=blue]{hyperref}
\usepackage{setspace}
\usepackage{amsrefs}
\usepackage{nicefrac}
\usepackage{graphicx}
\usepackage{color}
\usepackage{mathtools}

\numberwithin{equation}{section}

\newcommand{\ignore}[1]{}

\newcommand{\C}{{\mathbb{C}}}
\newcommand{\R}{{\mathbb{R}}}

\newcommand{\N}{{\mathbb{N}}}

\newcommand{\D}{{\mathbb{D}}}

\newcommand{\T}{{\mathbb{T}}}

\newtheorem{thm}{Theorem}[section]

\newtheorem{prop}[thm]{Proposition}

\newtheorem{lemma}[thm]{Lemma}

\newtheorem*{problem}{Problem}

\theoremstyle{definition}
\newtheorem{defn}[thm]{Definition}

\theoremstyle{remark}

\theoremstyle{plain} % just in case the style had changed
\newcommand{\thistheoremname}{}
\newtheorem{genericthm}[thm]{\thistheoremname}

\author{Alan Noell}
\address{Department of Mathematics, Oklahoma State University,
Stillwater, OK 74078, USA}
\email{noell@math.okstate.edu}

\date{\today}
\ifpdf
\hypersetup{
  pdftitle={Peak sets and boundary interpolation sets for the unit disc: a survey},
  pdfauthor={Alan Noell},
  pdfsubject={Complex Analysis},
  pdfkeywords={Peak sets, boundary interpolation sets, Lipschitz classes},
}
\fi

\title{Peak sets and boundary interpolation sets for the unit disc: a survey}

\keywords{Peak sets, boundary interpolation sets, Lipschitz classes}
\subjclass[2010]{30D50 (Primary),  30H05 (Secondary)}

\begin{document}

\def\s{\sigma}
\def\CC{{\mathbb C}^2}
\def\Cn{{\mathbb C}^n}
\def\L{\mbox{Lip }}
\def\b{\partial}
\def\B{{\mathcal B}}
\def\A{A_\ast}
\def\H{{\mathbb H}}
\def\BH{{\mathcal B}(\H)}

\begin{abstract}
This paper surveys results concerning
peak sets and boundary interpolation sets for the unit disc. It includes hitherto unpublished results  proved by David C. Ullrich
on peak sets for the Zygmund class.\end{abstract}

\dedicatory{To the memory of Thomas H. Wolff (1954--2000)}

\maketitle

\section{Introduction} \label{section:intro}

We denote the open unit disc in $\C$ by $\D$, and we denote its boundary by $\T$, the unit circle. If $E\subset \C$ is closed and $0< \alpha\leq 1$, we denote by $\Lambda_\alpha(E)$ the class of
functions $f$ on $E$ satisfying a Lipschitz condition of order $\alpha$ there: There exists a constant $C$ such that $|f(z)-f(w)|\leq C|z-w|^\alpha$ when $z,w\in E$.
If $0< \alpha\leq 1$ we denote by $A^\alpha$
the class of functions belonging to $\Lambda_\alpha(\overline{\D})$ and holomorphic on $\D$. We write $A^0$ or just $A$ for the disc algebra, the algebra of functions continuous on $\overline{\D}$ and holomorphic on $\D$. If $0\leq \alpha\leq 1$, we say that a closed subset $E$ of $\T$ is a peak
set for $A^\alpha$ if there exists a
function $f\in A^\alpha$ such that $f(z)=1$ when $z\in E$ and $|f(z)|<1$ when
$z\in \overline{\D}\setminus E$. We call $f$ a peak function. Clearly this
condition is equivalent to the existence of a strong
support function, that is, a function $g\in A^\alpha$ such that
$g(z)=0$ when $z\in E$ and $\mbox{Re }g(z)>0$ when $z\in\overline{\D}\setminus E$.
If $0\leq \alpha\leq 1$, we say that a closed subset $E$ of $\T$ is a boundary  interpolation
set for $A^\alpha$ if for each function $\phi\in \Lambda_\alpha(\T)$  there exists a
function $f\in A^\alpha$ such that $f(z)=\phi(z)$ when $z\in E$. (In the case $\alpha=0$ we require only that $\phi$ be continuous.)

In this paper we survey results concerning peak sets and boundary interpolation sets for $A^\alpha$. We also consider briefly related notions, such as zero sets and peak-interpolation sets, along with results for spaces having different regularity requirements at the boundary. Here is the major open problem.

\begin{problem}\label{question:pk} Characterize the peak sets for $A^\alpha$ when $0<\alpha <1$.
\end{problem}

We will see in Section \ref{section:lip} that there is no characterization in terms of a metric condition on the set, that is, a condition expressed in terms of the lengths of the arcs complementary to the set.

In Section \ref{section:prelim} we provide background information and explain why peak sets for $A^1$ are finite. In Section \ref{section:lip} we survey results on peak sets for $A^\alpha$ when $0<\alpha<1$, and in Section \ref{section:zyg} we present the proof by David C. Ullrich that peak sets for the Zygmund class are finite. Section \ref{section:interp} is a survey of results on boundary interpolation sets for $A^\alpha$ ($\alpha>0$) and some other spaces. In Section \ref{section:discalg} we consider the disc algebra, and in Section \ref{section:misc} we compare peak sets and boundary interpolation sets, especially for $A^\alpha$ when $0<\alpha <1$.

Thanks go to David C. Ullrich for granting permission to present his unpublished results on peak sets for the Zygmund class.

\section{Preliminaries}\label{section:prelim}

The letter $C$ denotes a fixed constant whose value may change from one line to the next. We write $u \lesssim v$ to mean $u\leq Cv$, $u \gtrsim v$ to mean $u\geq Cv$, and  $u\approx v$ to mean that both $u \lesssim v$ and $u \gtrsim v$.
If $E$ is a compact subset of $\C$, for $z\in \C$ we denote by $d_E(z)$ the distance from $z$ to $E$. We denote the normalized Lebesgue measure of a measurable subset $U$ of $\T$ by $m(U)$ or $|U|$.

Recall that by a classical result of F. and M. Riesz \cite{fmriesz} the only bounded holomorphic function on $\D$ whose radial limits vanish on a set of positive (Lebesgue) measure is the zero function. Every peak set for a space is the zero set of a function belonging to the same space, so for each space we consider every peak set has measure zero. The same conclusion holds for boundary interpolation sets:
If $E\subset\T$ is a boundary  interpolation set, find a function $f$ that interpolates $\phi(z)=1/z$ on $E$. Then the function $1-zf(z)$ is 0 on $E$ and is 1 at 0, so $E$ has measure zero.

Now we give a necessary condition on $d_E$ when $E\subset \T$ is closed and contained in the zero set of a nonconstant function $f$ in $A^\alpha$, where $0<\alpha\leq 1$. As we have just seen, this condition is satisfied if $E$ is a peak set or a boundary interpolation set for $A^\alpha$.
If such a function $f$ exists, then $\log|f|$ is integrable on $\T$ (Garnett \cite{garn}*{Theorem II.4.1}), and $|f|\leq C d_E^\alpha$ there. Taking the logarithm of each side of the inequality, we conclude that $\log d_E$ is integrable on $\T$. Here is an alternative formulation of this necessary condition: Write $\T\setminus E$ as the disjoint union of a sequence of open arcs, say $\{I_n\}$. Then this condition on $d_E$ says that $\sum_n |I_n|\log(1/|I_n|)$ converges. For these results, see Beurling \cite{beurling}
 and  Carleson \cite{carlzlip}.
We remark that this necessary condition on $E$ is also a sufficient condition for a closed set $E\subset \T$ of measure zero to be the zero set of some $f\in A^\alpha$. In fact, Carleson \cite{carlzlip} proved that one can choose $f$ to have derivatives up to the boundary of any given finite order, and that result was later extended to provide an infinitely differentiable $f$. See Korolevi\v{c} \cite{koro}, Novinger \cite{novz}, Taylor-Williams \cite{taywilze}, and the references therein.

Next we derive a stronger necessary condition on $d_E$ for peak sets. Let $E$ be a peak set for $A^\alpha$, where $0<\alpha\leq 1$, and let $g\in A^\alpha$ be a strong support function. On $\overline{\D}\setminus E$ define $H=1/g$, so $H$ is holomorphic on $\D$ and $\mbox{Re }H>0$. It follows from the Riesz-Herglotz representation theorem that there exist a finite positive Borel measure $\mu$ on $\T$ and a real constant $\beta$ such that, for $z\in \D$,
$$H(z)=i\beta+ \int_{\T}\frac{\zeta+z}{\zeta-z} \; d\mu(\zeta).$$
If $h\; dm$ is the absolutely continuous part of $\mu$ and $\tilde{\mu}$ denotes its Hilbert transform on $\T$, then
\begin{equation} \label{eq:basic}
H=h+i\tilde{\mu}+i\beta \text{ \ \ \ $m$-a.e.\ on $\T$.}
\end{equation}
Furthermore, by Kolmogorov's theorem $\tilde{\mu}$ satisfies a weak-type estimate: For all $\lambda>0$, $$|\{\zeta\in \T\colon |\tilde{\mu}(\zeta)|>\lambda\}|\leq C/\lambda.$$
 Good references for these results are Garnett \cite{garn}*{Sections I.3, I.5, III.1--2} and Koosis \cite{koos}*{Sections I:D--E, V:C}. (Sometimes the results
 are stated only in the case that $\mu$ is absolutely continuous with respect to $m$.)

These results provide information about $d_E$. Combining equation \eqref{eq:basic} with the fact that $|g|\leq C d_E^\alpha$ on $\T$, we obtain:

\begin{prop}\label{prop:bound}

Let $E$ be a peak set for $A^\alpha$, where $0<\alpha\leq 1$. Then there exist a finite positive Borel measure $\mu$ on $\T$ and a constant $C$ such that on $\T$ we have $$d_E^{-\alpha}\leq C(1+h+|\tilde{\mu}|).$$ Here $h \; dm$ is the absolutely continuous part of $\mu.$

\end{prop}

By the weak-type estimate for $\tilde{\mu}$ and Proposition \ref{prop:bound}, $d_E^{-\alpha}$ satisfies a weak-type estimate. Hence, for all $\epsilon>0$ we have  $$|\{\zeta\in \T\colon d_E(\zeta)\leq \epsilon\}|\leq C\epsilon^\alpha.$$  See Pavlov \cite{pavlov}, Hutt \cite{hutt}, Dyn{\textsc{\char13}}kin \cite{dynkinpk}, Bruna \cite{brunapk},  and
Lef\`{e}vre at al.\ \cite{lef}*{Proposition 2.8}.

In the case $\alpha=1$ it follows that $E$ is finite. (In addition to the preceding references, see Novinger-Oberlin \cite{novobe}, where an explicit bound for the cardinality of $E$ in terms of a strong support function is given. The weak-type estimate also gives a bound.)
In Section \ref{section:zyg} we present the stronger result due to Ullrich that peak sets are finite if we replace $\Lambda_1$ by the Zygmund class.

Next we discuss briefly the class of functions that are  continuously differentiable up to the boundary. By the preceding result, peak sets in this setting are finite, but here a somewhat more elementary argument is available. It depends on the Hopf lemma for subharmonic functions, a result that is useful for the study of peak sets in $\Cn$.
For a proof see, Forn{\ae}ss-Stens{\o}nes \cite{fost}*{Proposition 12.2}.

\begin{prop}[Hopf lemma]\label{thm:hopf}  Let $\Omega$ be a bounded domain in $\C$ whose boundary is twice continuously differentiable, and let $u$ a negative subharmonic
function on $\Omega$. Then there exists a constant
$C>0$ such that $u(z)\leq -Cd_{\b\Omega} (z)$ for all $z\in \Omega$.
\end{prop}

Now let $E$ be a peak set with a strong support function $g$ such that $g'\in A$. If $p\in E$ then by the Hopf lemma the derivative at $p$ of $\mbox{Re }g$ in the direction of
a normal to $\T$ is nonzero. Thus $g'(p)\neq 0$, and it follows that $p$ is an isolated point of $E$. This holds for all $p\in E$, so $E$ is finite. This result was proved in
Taylor-Williams \cite{taywilpk} (see also Caveny-Novinger \cite{cavnov}).

We conclude this section by stating a strong form of the converse: Every finite subset of $\T$ is a peak set having a peak function that is rational. This was proved by Taylor and Williams \cite{taywilpk}, who
in fact proved that each finite set is a  peak-interpolation set, in the following sense:

\begin{thm} \label{thm:pkint} Let $E$ be a finite subset of $\T$, and let $\phi$ be a function on $E$ such that $M=\max\{|\phi(\zeta)|\colon \zeta\in E\}$ is positive. Then there exists a rational function $f$ with poles outside $\overline{\D}$ such that $f=\phi$ on $E$ and $|f|<M$ on $\overline{\D}\setminus E$.
\end{thm}

The proof in \cite{taywilpk} gives an explicit construction of $f$. Of course, taking $\phi=1$ gives a peak function.

\section{Peak sets for $A^\alpha$}\label{section:lip}

In this section we study peak sets for $A^\alpha$ when $0<\alpha<1$. Some elementary properties of peak sets for $A^\alpha$ are easily stated: Closed subsets and finite unions of peak sets are peak sets (see Noell-Wolff \cite{nowol}). No satisfactory characterization of peak sets has been found.

As we saw in Section \ref{section:prelim}, if $E$ is a peak set for $A^\alpha$ then $d_E^{-\alpha}$ satisfies a weak-type estimate. In preparation for the use of Lorentz spaces below, we write this condition as $d_E^{-\alpha}\in L^{ 1 ,  \infty}(\T)$. Several results for this problem give necessary or sufficient conditions in terms of global conditions on $d_E$ or, equivalently, in terms of the lengths of the complementary arcs $\{I_n\}$ (see Section \ref{section:prelim}). We begin by surveying such results. In the course of our discussion we will see that no characterization of peak sets in terms of these lengths is possible, and we will briefly discuss nonglobal conditions.

Novinger and Oberlin \cite{novobe} proved that a necessary condition for $E$ to be a peak set for $A^\alpha$ is that $\sum_n |I_n|^{1-\alpha}|\log(1/|I_n|)|^{-\delta}$ converge for all $\delta >1$. They proved that the convergence of $\sum_n |I_n|^{(1-\alpha)/(3-\alpha)}$ is a sufficient condition and conjectured that a necessary and sufficient condition is that $\sum_n |I_n|^{1-\alpha}$ converge. Note that their sufficient condition is equivalent to $d_E^{-2/(3-\alpha)}\in L^1(\T)$, and the one in their conjecture is equivalent to $d_E^{-\alpha}\in L^1(\T)$.

We noted in Section \ref{section:prelim} that Bruna \cite{brunapk} showed that $d_E^{-\alpha}\in L^{ 1 ,  \infty}(\T)$ is a necessary condition for $E$ to be a peak set for $A^\alpha$. (As Bruna remarks, the necessary condition given in Novinger-Oberlin \cite{novobe}  follows from this.) Along the lines of the conjecture in \cite{novobe}, he  proved that a sufficient condition is that $d_E^{-\alpha}\in L^{1+\delta}(\T)$ for some $\delta >0$. This result follows from a nonglobal sufficient condition (see inequality \eqref{eq:bruna} below) for peak sets. Bruna uses methods of real analysis such as properties of functions of bounded mean oscillation and Muckenhoupt weights.

An improved  global sufficient condition follows from the work of Hru\v{s}\v{c}\"{e}v \cite{hruscev} and Dyn{\textsc{\char13}}kin \cite{dynkinpk}. Hru\v{s}\v{c}\"{e}v studied the Gevrey class $G_\alpha$ ($0<\alpha < 1$), which consists of all $f$ whose derivatives of all orders belong to $A$ and satisfy an estimate of the form $$|f^{(n)}|\leq R^{n+1} n!^{1+1/\alpha}.$$ He characterized the sets of nonuniqueness for $G_\alpha$, that is, the closed sets $E\subset \T$ for which there exists a nonconstant function in $G_\alpha$ that, together with all of its derivatives, vanishes on $E$.
Dyn{\textsc{\char13}}kin used Hru\v{s}\v{c}\"{e}v's work to show that every set of nonuniqueness for $G_\alpha$ is a peak set for $A^\alpha$. Now, Hru\v{s}\v{c}\"{e}v   gives the convergence of
$$\sum_n |I_n|^{1-\alpha}(\log(1/|I_n|))^{\alpha+\epsilon}$$ for some $\epsilon>0$ as a sufficient condition for a set of nonuniqueness for $G_\alpha$. (Hru\v{s}\v{c}\"{e}v attributes this result to S. A. Vinogradov and proves a stronger version \cite{hruscev}*{page 270}.) By Dyn{\textsc{\char13}}kin's result, this is a sufficient condition to be a peak set for $A^\alpha$.

 Dyn{\textsc{\char13}}kin \cite{dynkinpk} also conjectured that every peak set for $A^\alpha$ is a set of nonuniqueness for $G_\alpha$ (the converse of the result he proved). This conjecture would entail (\cite{hruscev}*{page 254}) that
$d_E^{-\alpha}\in L^1(\T)$ is a necessary condition for $E$ to be a peak set for $A^\alpha$. \"Erikke \cite{erikke} disproved this conjecture by constructing, for each $\alpha\in(0,1)$, a peak set $E$ for $A^\alpha$ such that $d_E^{-\alpha} \not\in L^1(\T)$. The set $E$ is of the form $\{1\}\cup\{\exp(ia_n)\colon n=1, 2, \ldots \}$, where $\{a_n\}$ is an increasing sequence in $(-\pi,0)$ with limit $0$. Independently, in  Noell-Wolff \cite{nowol} it is also proved that $d_E^{-\alpha}\in L^1(\T)$ is not a necessary condition (see Theorem \ref{thm:sharp}(b) below).

Noell and Wolff \cite{nowol} proved that $d_E^{-\alpha}\in L^1(\T)$ is a sufficient condition for $E$ to be a peak set for $A^\alpha$ and gave as a necessary condition that $d_E^{-\alpha}$ belong to a certain Lorentz space. Both of these results are sharp as global conditions expressed in terms of Lorentz spaces. They also characterized (in terms of Lorentz spaces) peak sets among Cantor sets with variable ratio of dissection. To state these results we first define the Lorentz spaces $L^{1,q}$ on a measure space $(X,\mu)$. (One can define $L^{p,q}$ for $0<p<\infty$, but we do not need these spaces.)  For $0<q<\infty$ we define $L^{1,q}(X,\mu)$, or just $L^{1,q}(X)$, to be the set of complex-valued measurable functions $f$ on $X$ for which $$\int_0^\infty  (\lambda \mu\{x\in X\colon |f(x)|> \lambda\})^{q} \frac{d\lambda}{\lambda}<\infty.$$ Thus $L^{1,1}(X)=L^{1}(X)$. The elements of $L^{1,\infty}(X)$ are those functions $f$ satisfying a weak-type estimate: $$\sup_{\lambda > 0} \lambda \mu\{x\in X\colon |f(x)|> \lambda\}<\infty.$$ Good references for Lorentz spaces in the context of interpolation spaces are Bergh-L\"ofstr\"om \cite{belo} and Grafakos \cite{graf}.

The following three theorems state the main results of \cite{nowol} on global conditions for peak sets.

\begin{thm}\label{thm:char} Let $E$ be a closed subset of $\T$ of measure zero, and let $\alpha\in(0,1)$.
\begin{enumerate}[(a)]
\item If $d_E^{-\alpha}\in L^1(\T)$ then $E$ is a peak set for $A^\alpha$.
\item If $E$ is a peak set for $A^\alpha$ then $d_E^{-\alpha}\in L^{1,1/\alpha}(\T)$.
\end{enumerate}
\end{thm}

The next result shows that the conditions of Theorem \ref{thm:char} are sharp in terms of Lorentz spaces. We use the Lorentz space $l^{1,q}$ obtained by taking $(X,\mu)$ to be  $\N$ with counting measure. If $E$ is a closed subset of $\T$ of measure zero then,  for $0<q<\infty$, $d_E^{-\alpha}\in L^{1,q}(\T)$
if and only the lengths of the complementary arcs satisfy $\{|I_n|^{1-\alpha}\}\in\l^{1,q}$. We denote by $\tau(E)$ the nonincreasing rearrangement of the sequence $\{|I_n|\}$.

\begin{thm}\label{thm:sharp} Let $\alpha\in(0,1)$, and assume that $\{c_n\}$ is a nonincreasing sequence of positive numbers such that $\sum_n c_n=1$.
\begin{enumerate}[(a)]
\item Suppose $\{c_n^{1-\alpha}\}\not\in l^1$ and $\{c_n\}$ contains a subsequence $\{c_{n_k}\}$ such that $c_{n_k}^{1-\alpha}\approx 2^{-k}$. Then there exists a closed set $E\subset\T$ of measure zero such that $\tau(E)=\{c_n\}$ and $E$ is not a peak set for $A^\alpha$.
\item Suppose $\{c_n^{1-\alpha}\}\in l^{1,1/\alpha}$. Then there exists a peak set $E$ for $A^\alpha$ such that $\tau(E)=\{c_n\}$ .
\end{enumerate}
\end{thm}

One consequence of this result is that the sequence $\tau(E)$ does not in general determine whether a closed set $E\subset\T$ of measure zero is a peak set for $A^\alpha$: For each $\alpha\in(0,1)$ there exist two closed subsets of $\T$ having measure zero such that one is a peak set for $A^\alpha$, the other is not, and the two sequences of lengths of complementary arcs are rearrangements of each other.

In the case of Cantor sets with variable ratio of dissection there is a metric characterization of peak sets. Here is the construction: Fix a sequence $\{\sigma_n\}_{n=1}^\infty$ in $(0,\frac{1}{2})$ with limit $\sigma$ such that $\sigma\in (0,\frac{1}{2})$.  Put $\delta_0=1$ and, for $n\geq 1$, $\delta_n=\prod_{k \leq n} \sigma_k$. Apply the Cantor construction on $[0,1]$ using $\sigma_n$ as the dissection ratio at the $n$th stage to obtain a closed set $E$ of measure zero, and identify $E$ with a subset of $\T$. We say that $E$ has parameter $\sigma$. Put $\alpha=1-\log 2/\log(1/\sigma)$.
The complement of $E$ is the union over $n\in \N$ of $2^{n-1}$ open arcs, each of length $\delta_{n-1}(1-2\sigma_n)$.
If $\beta<\alpha$ then considering the complementary arcs shows that $d_E^{-\beta}\in L^1(\T)$, so $E$ is a peak set for $A^\beta$. If $\beta >\alpha$ then $d_E^{-\beta}\not\in L^{1,\infty}(\T)$, so
$E$ is not a peak set for $A^\beta$. Thus, the interesting case for peak sets is  $\beta=\alpha$.

\begin{thm}\label{thm:cantor} Let $E$ be a Cantor set with variable ratio of dissection having parameter $\sigma$, and put $\alpha=1-\log 2/\log(1/\sigma)$. Then $E$ is a peak set for $A^\alpha$ if and only if $d_E^{-\alpha}\in L^{1,2/(1+\alpha)}(\T)$.
\end{thm}

Note that this Lorentz space is halfway between that for the necessary condition in Theorem \ref{thm:char} and that for the sufficient condition. An equivalent formulation of the condition $d_E^{-\alpha}\in L^{1,2/(1+\alpha)}(\T)$ is that $\{2^n\delta_n^{1-\alpha}\}\in l^{2/(1+\alpha)}$. For example, the usual middle-thirds Cantor set is not a peak set for $A^\alpha$ if $\alpha=1-\log 2/\log 3$.

Now we discuss briefly the proofs of the first two of these results, sketching the proofs of  only Theorem \ref{thm:char}(a) and Theorem \ref{thm:sharp}(b). Fix $\alpha\in(0,1)$. We follow \cite{nowol} by working on the upper half plane $\H$ and write $A^\alpha(\H)$ for the corresponding space. Both Theorem \ref{thm:char}(a) and Theorem \ref{thm:sharp}(b)  assert the existence of peak functions. In each case, the procedure used for constructing such functions depends on the following.

\begin{prop}\label{prop:tool}
Let $E$ be a compact subset of $\R$. Assume that there exists a finite positive Borel measure $\mu$ on $\R$ such that, if $h \; dx$ is the absolutely continuous part of $\mu$, then $h$ is smooth (infinitely differentiable) on $\R\setminus E$, the mutually singular part of $\mu$ is supported on $E$, and on $\R\setminus E$ we have
$$d_E^{-\alpha}\leq C(1+h+|\tilde{\mu}|)$$
and \begin{equation} \label{eq:bound} |h'|+|\tilde{\mu}'|\leq C(1+h+|\tilde{\mu}|)^{1+1/\alpha} \end{equation}
(Here $\tilde{\mu}$ is the Hilbert transform of $\mu$ on $\R$.)
Then $E$ is a peak set for $A^\alpha(\H)$.
\end{prop}

\begin{proof}[Sketch of proof]
For $z\in \H$ define $$H(z)=\frac{1}{\pi i} \int_{\R} \frac{1}{\zeta-z} \; d\mu(\zeta).$$ Then $H$ is holomorphic, $\mbox{Re } H> 0$, and $f=H/(1+H)$ extends continuously to $\overline{\H}$ as a peak function for $E$ that is smooth on $\R\setminus E$. One checks using inequality \eqref{eq:bound} that $f\in A^\alpha(\H)$.
\end{proof}

Note that the converse of Proposition \ref{prop:bound} is false, but with its additional hypotheses Proposition \ref{prop:tool} serves as a converse for the arguments that follow.\\

\noindent {\it Sketch of proof of Theorem \ref{thm:char}(a).} We construct an integrable function $h$ such that Proposition \ref{prop:tool} is satisfied with $\mu=h \; dx$.  By hypothesis, $d_E^{-\alpha}$ is locally integrable on $\R$. We smooth and truncate this function on each component of the complement of $E$ to get a smooth nonnegative integrable function $u$ on $\R\setminus E$,  with control on the derivatives and the Hilbert transform $\tilde{u}$. Then $u$ satisfies $|u'|\leq C(1+u)^{1+1/\alpha}$, but the bound on $|\tilde{u}'|$ needed to verify inequality \eqref{eq:bound} may fail. To account for this, first we get control on the size of intervals where this bound is a problem. Then on each such interval we add to $u$ a smooth $H^1$ atom. (Recall that an $H^1$ atom is an integrable function $v$ having mean value zero such that there exists an interval $(a,b)$ containing the support of $v$  and  $|v|\leq 1/(b-a)$.)
The result of adding these atoms to $u$ is the function $h$ with the desired properties. \hfill\qedsymbol\\

\noindent {\it Sketch of proof of Theorem \ref{thm:sharp}(b).} We find a peak set $E$ for $A^\alpha(\H)$ such that the nonincreasing rearrangement of the lengths of the bounded components of the complement of $E$ equals the given nonincreasing sequence $\{c_n\}\in l^{1}$ of positive numbers satisfying $\{c_n^{1-\alpha}\}\in l^{1,1/\alpha}$. We take $E= \{0\}\cup \{a_n\}$, where $a_n=\sum_{j=n}^\infty c_j$. Note that $\{a_n\}$ is a decreasing convex sequence with limit $0$. It is not hard to see that there exists $\delta>0$ such that $16\delta a_n^{1/\alpha}\leq c_n$ for all $n$. It turns out (\cite{nowol}*{Lemma 1.2}) that the condition $\{c_n^{1-\alpha}\}\in l^{1,1/\alpha}$ is equivalent to the condition $\sum_n a_n^{-1+1/\alpha}<\infty$.  The first step in constructing a measure $\mu$ so that Proposition \ref{prop:tool} holds is to use $\mu_0$, the point mass at $0$. Because $\tilde{\mu_0}(x)=1/(\pi x)$, the inequality $d_E^{-\alpha}\leq C(1+|\tilde{\mu_0}|)$ fails on certain intervals centered at the $a_n$. To compensate we add $h \; dx$ to $\mu_0$, where $h$ is an $L^1$ function that is bounded below by a constant multiple of $d_E^{-\alpha}$ on these intervals. It suffices that for all $n$ we have $h(x)\gtrsim |x-a_n|^{-\alpha}$ when $|x-a_n|<\delta a_n^{1/\alpha}$, and summing over $n$ the integral of $|x-a_n|^{-\alpha}$ over this interval yields a constant multiple of $\sum_n a_n^{-1+1/\alpha}$. Thus, the existence of an $L^1$ function $h$ with the required property is implied by the assumption $\{c_n^{1-\alpha}\}\in l^{1,1/\alpha}$---in fact, the two conditions are equivalent. For an appropriate choice of $h$ the measure $h \; dx+\mu_0$ satisfies in addition inequality \eqref{eq:bound} of Proposition \ref{prop:tool}. \hfill\qedsymbol\\

We remark that the preceding construction of the measure $h \; dx+\mu_0$ suggests the necessary condition $\{c_n^{1-\alpha}\}\in l^{1,1/\alpha}$ for peak sets (at least those of this form): Given a compact set $E$  defined in terms of a decreasing convex sequence  $\{a_n\}$ of positive numbers and its limit $0$, the convergence of $\sum_n a_n^{-1+1/\alpha}$ arises as a necessary condition when as above we try to add to $\mu_0$ an $L^1$ majorant of $d_E^{-\alpha}$ near each $a_n$. To make this argument rigorous, one can apply a localization result, Proposition  \ref{prop:local} below, to sets $E$ of this form and thus obtain the necessary condition $\{c_n^{1-\alpha}\}\in l^{1,1/\alpha}$.
The assertion in Theorem \ref{thm:char}(b) that this condition is necessary for general peak sets is more complicated to prove. Similarly, the proof of Theorem \ref{thm:sharp}(a) involves some technicalities, but  Proposition \ref{prop:local} is central to the argument.

For the statement of this proposition, we alter our earlier convention: We write $|U|$ for the Lebesgue measure of a measurable set $U\subset \R$. Given an interval $I\subset \R$ and $\delta >0$, we write $\delta I$ for the interval with the same center as $I$ and length $\delta|I|$. Given $N\in \N$, $\delta>1$, and a collection $\mathcal{I}$ of intervals, we say that $\mathcal{I}$ is $(N,\delta)$-disjoint if no point of $\R$ belongs to more than $N$ intervals of the form $\delta I$ with $I\in \mathcal{I}$. We write $||\mu||$ for the total variation of the finite signed Borel measure $\mu$.

\begin{prop}\label{prop:local}
Given $N\in \N$ and $\delta>1$, there exist finite constants $C_1$ and $C_2$ such that the following holds: Let $\mu$ be a finite signed Borel measure on $\R$. If $h \; dx$ is the absolutely continuous part of $\mu$, define $F=|h|+|\tilde{\mu}|$. Let $\mathcal{I}$ be any $(N,\delta)$-disjoint collection of intervals, and suppose that for every $I\in\mathcal{I}$ a number $\lambda_I>C_1 ||\mu||/|I|$ is given. Then
\begin{equation} \label{eqn:loc} \sum_{I\in\mathcal{I}} \lambda_I |\{x\in I\colon F(x)>\lambda_I\}|\leq C_2 ||\mu||.\end{equation}
\end{prop}

This proposition is used as a local version of the weak-type estimate derived in Section \ref{section:prelim}. A typical application (e.g., in the proof of Theorem \ref{thm:sharp}(a)) is to show that a compact set $E\subset \R$ is not a peak set for $A^\alpha(\H)$ by showing that the conclusion of Proposition \ref{prop:bound} (with $\T$ replaced by $\R$) cannot hold: For some choice of $\mathcal{I}$, the sum obtained by replacing $F$ by $d_E^{-\alpha}$ in the left-hand side of inequality \eqref{eqn:loc} diverges.

Now we focus on nonglobal conditions for peak sets.
Before the results in \cite{nowol}, Bruna \cite{brunapk} had already questioned whether peak sets for $A^\alpha$ could be characterized by a global condition. One nonglobal sufficient condition he gives is the existence of an $L^1$ majorant $\phi$ of $d_E^{-\alpha}$ on $\T$ such that for all $z\in \T\setminus E$ we have
\begin{equation} \label{eq:bruna}
d_E(z)\int_{\T\setminus J_z} \frac{\phi(\zeta)}{|\zeta-z|^2} \; dm(\zeta) \leq C \phi(z).
\end{equation}
Here $J_z=\{\zeta\in\T\colon |\zeta-z|\leq d_E(z)/4\}$.
In part, Bruna's doubt about the possibility of a global characterization was based on the work of Hru\v{s}\v{c}\"{e}v \cite{hruscev} described above, who shows that there is
no global condition characterizing sets of nonuniqueness for the Gevrey classes. Hru\v{s}\v{c}\"{e}v does characterize the sets of nonuniqueness for $G_\alpha$ by the following condition: There exists an $L^1$ majorant $\phi$ of $d_E^{-\alpha}$ on $\T$ such that for all $z\in \T\setminus E$ we have $$\int_{\T\setminus I_z} \frac{\phi(\zeta)}{|\zeta-z|^2} \; dm(\zeta) \leq C \phi(z)^{1+1/\alpha}.$$ Here $I_z$ is the complementary arc containing $z$. (This condition is implied by Bruna's nonglobal sufficient condition for peak sets defined by inequality \eqref{eq:bruna} above.) Of course a condition requiring that $d_E^{-\alpha}$ have a majorant in $L^1$ cannot characterize peak sets, but these two conditions are illustrative.

Bomash \cite{bomash} proved the following: A closed set $E\subset \T$ is a peak set for $A^\alpha$ if and only if there exists a finite positive Borel measure $\mu$ on $\T$ such that $|h+i\tilde{\mu}|^{-1}$ is equal almost everywhere on $\T$ to a function in $\Lambda_\alpha(\T)$ that equals zero on $E$. Here $h \; dm$ is the absolutely continuous part of $\mu$.  (Compare this result with
the discussion in Section \ref{section:prelim} leading to equation \eqref{eq:basic}.) The key part of the proof is to show that if $f\in A$ and $\mbox{Re } f\geq 0$ then $f\in A^\alpha$ if and only if $|f|\in \Lambda_\alpha(\T)$. Bomash also gave a sufficient condition for peak sets expressed in terms of an integral inequality. (It implies that $d_E^{-\alpha}\in L^1(\T)$.) It is not hard to see using the work of Hru\v{s}\v{c}\"{e}v \cite{hruscev} that every nonuniqueness set for $G_\alpha$ satisfies this condition.

We mention a few connections between peak sets and other areas. One of the original motivations for the study of peak sets for $A^\alpha$ was the analysis of the singular spectrum of the self-adjoint Friedrichs model in perturbation theory. See Pavlov \cite{pavlov} and Faddeev-Pavlov  \cite{fadpav}.
 In another direction,  Hamilton \cite{hamilton} studied holomorphic conformal flows and applied his results to peak sets for functions having a modulus of continuity $w(t)$ satisfying $\int_0^1 1/w(t) \; dt=\infty$. He shows that such a peak set is a fixed set for some flow and concludes that it has logarithmic capacity zero. Noell and Wolff \cite{nowol} studied size properties of Cauchy integrals of measures. They define the space $X$ of all real-valued functions $\phi$ on $\R$ for which there exists a finite positive Borel measure $\mu$ on $\R$ such that, if $h \; dx$ is the absolutely continuous part of $\mu$, then $|\phi| \leq |h+i\tilde{\mu}|$. By Proposition \ref{prop:bound}, if $E$ is a peak set for $A^\alpha(\H)$ then $d_E^{-\alpha}$ belongs locally to $X$, in the sense that there exists $\phi\in X$ such that $\phi=d_E^{-\alpha}$ on an open set containing $E$. They prove that, if $E\subset\R$ is compact and $d_E^{-\alpha}$ belongs locally to $X$, then $d_E^{-\alpha}$ belongs locally to $ L^{1,q}(\R)$, where $q=\max(2,1/\alpha)$. They give a characterization when $E$ is a Cantor set with variable ratio of dissection.

For results about peak sets for Lipschitz classes in $\Cn$, see Ababou-Boumaaz \cite{abab}, Saerens \cite{saer}, Cascante \cite{casc}, and Noell \cite{noell}.

\section{Peak sets for the Zygmund class}\label{section:zyg}

In this section we consider peak sets for the space of elements of $A$ having boundary values in the Zygmund class, a space that lies between $A^1$ and $A^\alpha$ for $\alpha<1$. David C. Ullrich proved (personal communication) that these peak sets are finite, just as for $A^1$. Because Ullrich never published this result, we provide much of the proof.

Recall that the Zygmund class $\Lambda_\ast=\Lambda_\ast(\T)$ is the set of all functions $f$ continuous on $\T$ for which $|f(e^{it}\zeta)+f(e^{-it}\zeta)-2f(\zeta)|\leq Ct$ for all $t>0$ and all $\zeta\in\T$. Define $\A$ to be the set of all $f \in A$ such that $f|_{\T} \in \Lambda_\ast$. Ullrich's proof that peak sets for $\A$ are finite depends on the characterization of $\A$ in terms of the Bloch space. Recall that the Bloch space $\B$ for $\D$ is the set of all functions $g$ holomorphic on $\D$ for which $|g'(z)|(1-|z|)$ is bounded on $\D$. Zygmund proved (Pommerenke \cite{pomm}*{Section 7.2}) that $\A$ consists of all primitives of Bloch functions: If $f$ is holomorphic on $\D$, then $f$ extends to be an element of $\A$ if and only if $f'\in \B$.

The proof of Ullrich's result is simplified by working on the upper half-plane $\H$. We write $\A(\H)$ for the corresponding space. The derivative of an element of $\A(\H)$ belongs to
$\BH$, the Bloch space on $\H$. To be explicit, $g\in \BH$ says that the  quantity  $$||g||_{\BH}=\sup\{|g'(z)|\; \mbox{Im } z\colon z\in\H\}$$ is finite. The main tools for Ullrich's proof are the Julia-Carath\'{e}odory theorem and arguments of Ullrich in \cite{ullrichlon}. In a sense, his proof is an elaboration of the argument using the Hopf lemma in Section \ref{section:prelim}: The key point is to control difference quotients along the real line. A couple of preliminary results are necessary.

\begin{lemma}\label{lemma:estbl}

There exists a constant $C$ such that for every $g\in\BH$ we have $$|g(x_1+iy_1)-g(x_2+iy_2)|\leq C (1+\log(y_1/y_2))||g||_{\BH}$$ when $0<y_2\leq y_1$, $x_1$ and $x_2$ are real,  and $|x_1-x_2|\leq y_1$.

\end{lemma}

\begin{proof}[Sketch of proof]

The left-hand side is the absolute value of the integral from $x_2+iy_2$ to $x_1+iy_1$ of $g'$. Parametrize, move the absolute value inside the integral, estimate the integrand using the definition of $||g||_{\BH}$, and calculate the resulting integral. The process is simplified somewhat by first reducing to the case $x_1+iy_1=i$. This is possible because
  $||g||_{\BH}$ is invariant under horizontal translation and vertical dilation on $\H$.
\end{proof}

Next we prove a lemma relating the horizontal and vertical difference quotients at $0$ of $f\in \A(\H)$.

\begin{lemma}\label{lemma:estdif}

There exists a constant $C$ such that for every $f\in \A(\H)$ and every real $t\neq 0$ we have $$\left|\frac{f(i|t|)-f(0)}{i|t|}-\frac{f(t)-f(0)}{t}\right|\leq C ||f'||_{\BH}.$$

\end{lemma}

\begin{proof} We assume that $t>0$. Define $\gamma_t$ and $\omega_t$ from $[0,1]$ to $\overline{\H}$ by $\gamma_t(s)=its$ and $\omega_t(s)=\frac{t}{2}(1-e^{-i\pi s})$. Thus $\omega_t$ traces the semicircle meeting the real axis orthogonally at $0$ and $t$.

We claim that it suffices to prove the following inequalities: There exists a constant $C$ such that if $g\in \BH$ and $t>0$, then
$$\left|\frac{1}{it}\int_{\gamma_t} g(z)\; dz -g(it)\right|\leq C ||g||_{\BH}$$
and
$$\left|\frac{1}{t}\int_{\omega_t} g(z)\; dz -g(it)\right|\leq C ||g||_{\BH}.$$
In fact, given $f\in \A(\H)$ we need only apply these two inequalities with $g=f'$, integrate, and apply the triangle inequality to obtain the desired conclusion.

We sketch the proof of the first inequality. (The proof of the second is similar.) Note that
$$\frac{1}{it}\int_{\gamma_t} g(z)\; dz -g(it)=\frac{1}{it}\int_{\gamma_t} (g(z)-g(it))\; dz.$$
Now use the parametrization, apply Lemma \ref{lemma:estbl}, and integrate the resulting quantity.
 \end{proof}

 The following result is essentially a consequence of the arguments in Ullrich \cite{ullrichlon}, but for convenience we give a proof here. It uses the semicircular arc introduced in the proof of Lemma \ref{lemma:estdif}.

 \begin{lemma}\label{lemma:avg}

Assume that $g\in \BH$ and $g$ has nontangential limit $L$ at 0. For $t\neq 0$ define $\omega_t$ from $[0,1]$ to $\overline{\H}$ by $\omega_t(s)=\frac{t}{2}(1-e^{-i\pi s})$. Then
$$\lim_{t\rightarrow 0} \; \frac{1}{t}\int_{\omega_t} g(z)\; dz = L.$$

\end{lemma}

\begin{proof} For simplicity we assume that $L=0$ and restrict our attention to positive $t$. Fix a small positive number $\delta$. Split $\omega_t$ into two contours by defining $\alpha_t$ to be its restriction to $[0,1-\delta]$ and $\beta_t$ its restriction to $[1-\delta,1]$. Because $g$ has nontangential limit $0$ at $0$, we have
$$\lim_{t\searrow 0} \; \frac{1}{t}\int_{\alpha_t} g(z)\; dz = 0.$$

To estimate the integral over $\beta_t$ we write $p(t)=\omega_t(1-\delta)$ and $\lambda(t)=\int_{\beta_t}  dz$. Then $\lim_{t \searrow 0} g(p(t))=0$ because $g$ has nontangential limit $0$ at $0$. Comparing the chord length to the arc length gives
$$|\lambda(t)|=\frac{t}{2}|e^{i\pi\delta}-1|\leq \frac{\pi \delta t}{2}.$$ As in the proof of Lemma \ref{lemma:estdif}, there exists a constant $C>0$ independent of $\delta$ such that
$$\left|\frac{1}{\lambda(t)}\int_{\beta_t} g(z) \; dz - g(p(t)) \right|\leq C||g||_{\BH}.$$
Thus,
$$\left|\frac{1}{t}\int_{\beta_t} g(z) \; dz \right|\leq\frac{|\lambda(t)|}{t}(C||g||_{\BH}+|g(p(t))|)\leq \frac{\pi \delta}{2}(C||g||_{\BH}+|g(p(t))|).$$ Hence,
$$\limsup_{t\searrow 0} \left|\frac{1}{t}\int_{\beta_t} g(z) \; dz \right|\leq \frac{\pi \delta}{2} C||g||_{\BH}.$$
This proves the desired result.
 \end{proof}

Now we can prove the main result of this section.

\begin{thm}[Ullrich]\label{thm:ull}
Every peak set for $\A(\H)$ is finite.

\end{thm}

\begin{proof}
Let $E$ be a peak set for $\A(\H)$. We assume for simplicity that $0\in E$ and show that $0$ is an isolated point of $E$. Let $f\in\A(\H)$ be a peak function for $E$.
We use the Julia-Carath\'{e}odory theorem (Ahlfors \cite{ahlfors}*{Theorem 1-5} or Pommerenke \cite{pomm}*{Propositions 4.7 and 4.13}): The difference quotient
$\displaystyle \frac{f(z)-1}{z}$ has nontangential limit $L\in(0,\infty]$ for $z\rightarrow 0$, and if $L<\infty$ then $f'$ has nontangential limit $L$ at $0$.

In case $L=\infty$,
it follows from Lemma \ref{lemma:estdif} and the reverse triangle inequality that $$\left| \frac{f(t)-1}{t} \right|\geq  \left|\frac{f(i|t|)-1}{i|t|}\right| -C||f'||_{\BH}\rightarrow \infty$$ as $t\rightarrow 0$ (with $t$ real). Because $E=f^{-1}(1)$, we see in this case that $0$ is an isolated point of $E$, as desired.

Now assume that $L<\infty$,  so $f'$ has nontangential limit $L$ at $0$. We apply Lemma \ref{lemma:avg} with $g=f'$: As $t\rightarrow 0$ (with $t$ real),
$$\frac{f(t)-1}{t}=\frac{1}{t}\int_{\omega_t} f'(z)\; dz\rightarrow  L.$$ Because $L>0$, it follows that $0$ is an isolated point of $E$.
\end{proof}

Ullrich asked whether the cardinality of a peak set for $\A$ can be bounded in terms of the appropriate norm of a peak function. As we noted in Section \ref{section:prelim}, there is  such a bound for $A^1$.

\section{Boundary interpolation sets for $A^\alpha$}\label{section:interp}

In this section we present results about boundary interpolation sets for $A^\alpha$ if $\alpha>0$. We also list results for some other classes of functions. We remark that most of the interpolation problems in this section were originally studied for subsets of $\overline{\D}$ (not just subsets of $\T$), in which case the results and proofs are more complicated.

 Let $E$ be a closed subset of $\T$ with measure zero. In Section \ref{section:prelim} we saw that a necessary condition for interpolation is that $\log d_E$ be integrable on $\T$. That is a global condition, but in this section we will see a local condition on
 $d_E$ that is necessary and sufficient for interpolation. This condition was introduced by Koto\v{c}igov \cite{kotocigov} in the context of interpolation problems involving elements of $A$ having derivatives in a Hardy space.

\begin{defn} Let $E$ be a closed subset of $\T$. We say that $E$ is porous, or satisfies condition (K), if there exists a positive constant $C$ such that for every arc $I\subset \T$ we have $\sup\{d_E(z)\colon z\in I\}\geq C|I|$.
\end{defn}

Every porous set has measure zero. If (as in Section \ref{section:prelim}) we write $\T\setminus E$ as the disjoint union of a sequence of open arcs, say $\{I_n\}$, then porosity says the following: There exists a positive constant $C$ such that for every  arc $I\subset\T$ with endpoints in $E$ there exists $n$ such that $I_n\subset I$ and $|I_n|\geq C|I|$. Equivalent formulations of porosity can be found in Dyn{\textsc{\char13}}kin \cite{dynkin} (see below on interpolation for the Gevrey  classes) and Bruna \cite{brunaint}*{Theorem 3.1}.

Considered as a subset of $\T$, the Cantor set is porous (as are all Cantor sets with variable ratio of dissection). Other examples given by Dyn{\textsc{\char13}}kin \cite{dynkin} are of the form $E=\{1\}\cup\{\exp(i\psi(n))\colon n=1, 2, \ldots \}$. If $\psi$ is a decreasing exponential function, $E$ is porous; if $\psi$ is a decreasing power function, $E$ is not porous.

Dyn{\textsc{\char13}}kin \cite{dynkin} proved that porosity characterizes boundary interpolation sets for $A^\alpha$ if $0<\alpha<1$.

\begin{thm}[Dyn{\textsc{\char13}}kin]\label{thm:dyn}  Let $E$ be a closed subset of $\T$, and let $0<\alpha<1$. Then $E$ is an interpolation set for $A^\alpha$ if and only if $E$ is porous. \end{thm}

Note that the characterization does not depend on $\alpha$.
First we prove that porosity is necessary for interpolation. Then we make a few comments about proving sufficiency.

\begin{proof}[Proof of necessity of porosity]

Observe that if we can interpolate from a set then we can interpolate with bounds, in the following sense.
If $L\subset \C$ is a compact set and $\phi\in \Lambda_\alpha(L)$, we put $$||\phi||_{\Lambda_\alpha(L)}=\sup\{|\phi(z)|\colon z\in L\}+\sup\left\{\frac{|\phi(z)-\phi(w)|}{|z-w|^\alpha}\colon z,w\in L, z\neq w\right\}.$$
Then $\Lambda_\alpha(L)$ is complete in this norm. Further, $A^\alpha$ is a closed subspace of $\Lambda_\alpha(\overline{\D})$, so $A^\alpha$ is complete in the norm it inherits. Now let $E\subset\T$ be an interpolation set for $A^\alpha$. The bounded linear transformation that restricts an element of $A^\alpha$ to $E$ maps onto $\Lambda_\alpha(E)$, so by the open mapping theorem that transformation is open. Hence, there is a constant $C$ such that for every $\phi\in \Lambda_\alpha(E)$ there exists $f\in A^\alpha$ interpolating $\phi$ with $||f||_{A^\alpha}\leq C ||\phi||_{\Lambda_\alpha(E)}.$
We remark that the same argument applies to boundary interpolation sets for the disc algebra $A$.

Fix an interpolation set $E\subset \T$ for $A^\alpha$. Given an arc $I\subset \T$, we write $R_I$ for the curvilinear rectangle $\{z\in\overline{\D}\colon 1-|z| \leq |I|, z/|z|\in I\}$. We denote the center of $R_I$ by $p_I$, and we define $\phi_I(z)=1/(z-p_I)$ for $z\in E$. In what follows we drop the subscripts involving $I$, but all constants are independent of $I$.

We claim that $||\phi||_{\Lambda_\alpha(E)}\lesssim |I|^{-1-\alpha}$. To see this, note first that
$d_E(p)\gtrsim |I|$ (because $d_{\T} (p)\gtrsim |I|$),
 so $|\phi|\lesssim |I|^{-1}$. In light of this estimate, to complete the proof of the claim we may assume that two given distinct points
 $z,\zeta \in E$ satisfy $|z-\zeta|\leq |I|$. Then
$$\frac{|\phi(z)-\phi(\zeta)|}{|z-\zeta|^\alpha}=\frac{|z-\zeta|^{1-\alpha}}{|z-p||\zeta-p|}\lesssim |I|^{1-\alpha}|I|^{-2}=|I|^{-1-\alpha}.$$

It follows from the claim and the observation about interpolating with bounds that there exists $f\in A^\alpha$ such that $f=\phi$ on $E$ and $$||f||_{A^\alpha}\leq C ||\phi||_{\Lambda_\alpha(E)}\lesssim |I|^{-1-\alpha}.$$
Put $g(z)=1-(z-p)f(z)$, so $g(p)=1$ and $g=0$ on $E$. Porosity will follow from applying the Jensen formula on $R$ to $\log |g|$ using harmonic measure, but first we need to prove another claim.

We claim that $|g|\lesssim 1$ on $\partial R$ and $|g(z)|\lesssim (d_E(z)/|I|)^\alpha$ for $z\in I$. We prove only the second part, and in view of the definition of porous we may assume that $E\cap I$ is nonempty. Let $z\in I$. If $\zeta\in E\cap I$, then, because $|z-p|\lesssim |I|$ and $f(\zeta)=1/(\zeta-p)$,
$$|g(z)|=|z-p||f(z)-1/(z-p)|\lesssim |I|(|f(z)-f(\zeta)|+|1/(z-p)-1/(\zeta-p)|).$$ Now, using the preceding estimate for $||f||_{A^\alpha}$, we get $$|f(z)-f(\zeta)|\leq ||f||_{A^\alpha} |z-\zeta|^\alpha\lesssim |I|^{-1-\alpha}|z-\zeta|^\alpha.$$ A similar estimate holds for the other term. Because these estimates hold for all $\zeta\in E\cap I$, we conclude that, as claimed,
$$|g(z)|\lesssim |I| |I|^{-1-\alpha}d_E^\alpha (z) = (d_E(z)/|I|)^\alpha.$$

Now let $\omega$ denote the harmonic measure at $p$ relative to $R$ (\cite{krantz}*{Chapter 6} or \cite{pomm}*{Section 4.4}). Applying the Jensen formula, integrating over $\partial R\setminus I$ and $I$, and using the estimates in the immediately preceding claim, we get a constant $C_1$  such that
$$0=\log|g(p)|\leq \int_{\partial R} \log|g| \; d\omega\leq C_1+\alpha \omega(I) \log\frac{\sup_{z\in I} d_E(z)}{|I|}.$$ It is easy to see that $\omega(I)\gtrsim 1$, and it follows immediately that $\sup_{z\in I} d_E(z)\gtrsim |I|$.
\end{proof}

With regard to the sufficiency of porosity for interpolation, we state a general method of interpolation for functions with some degree of regularity, one that is useful in $\Cn$ as well.
For simplicity assume that the function $\phi$ we want to interpolate from a given porous set $E$ is smooth on $\T$. For a given function $\psi(z)$ we write
$\overline{\partial}\psi$ for $\frac{1}{2}({\partial \psi}/{\partial x}+i{\partial \psi}/{\partial y})$. First use Harvey-Wells \cite{harwel}*{Lemma 1.6} (which depends on  Whitney's extension theorem \cite{whitney}) to find a smooth extension $\Phi$ of $\phi$ such that
 $\overline{\partial}\Phi$ and its derivatives of all orders vanish on $E$. Next show that porosity allows the construction of a function $G$  that is holomorphic on $\D$, extends to be smooth and nonzero on $\overline{\D}\setminus E$, and vanishes on $E$ to a prescribed order, with control on the growth of its derivatives near $E$. Then find a solution $u$ of the equation $\overline{\partial}u=\overline{\partial}\Phi/G$ with control on the growth of $u$ and its partial derivatives near $E$; here one uses the standard integral solution of the $\overline{\partial}$-equation, porosity, and properties of $G$. Now if we define $f=\Phi-uG$ then $f$ is holomorphic on $\D$ (because $u$ solves the $\overline{\partial}$-equation and $G$ is holomorphic), $f$ agrees with $\phi$ on $E$, and the appropriate regularity of $f$ on $\overline{\D}$ follows from the control on the growth of the derivatives of $G$ and $u$. In the case of interpolation in $A^\alpha$ adjustments must be made. (See Bruna-Tugores \cite{brunatu}*{Section III} for a helpful statement.)  In outline, though, Dyn{\textsc{\char13}}kin's argument is similar to the above. The main difficulty lies in finding the function we have called $G$. Dyn{\textsc{\char13}}kin constructs an outer function with the required regularity, order of vanishing on $E$, and control on the derivatives. We omit the details of his careful construction. Bruna \cite{brunaint} gives an alternative proof of this result that also covers other cases (see below). He constructs the outer function by using the function $\log d_E$, which he proves to be of bounded mean oscillation when $E$ is porous.

As we have seen,  Dyn{\textsc{\char13}}kin characterized boundary interpolation sets for $A^\alpha$  as  porous sets when $0<\alpha<1$.  Several other settings for interpolation have been studied. Dyn{\textsc{\char13}}kin's original paper \cite{dynkin} considered, for $\alpha>0$ not integral, the space of functions whose first $\lfloor \alpha\rfloor$ derivatives belong to $A$ and for which the derivative of order $\lfloor \alpha\rfloor$ satisfies a Lipschitz condition of order $\alpha-\lfloor \alpha\rfloor$. (We are using the floor, or greatest integer, function $\lfloor \alpha \rfloor$.) He characterized as porous the boundary interpolation sets for these spaces. (Here, and in what follows, derivatives are interpolated as well as the function.)
Bruna \cite{brunaint} extended this result to the case when $\alpha$ is a positive integer: Porosity characterizes the boundary interpolation sets for the space of functions whose derivatives through order $\alpha$ belong to $A$. Bruna and Tugores \cite{brunatu} proved that the same characterization holds for $A^1$.

The characterization is different for the space $A^\infty$ of functions whose derivatives of all orders belong to $A$. Alexander, Taylor, and Williams \cite{atw} proved that the boundary interpolation sets for $A^\infty$ are characterized by the following condition: There exist constants $C_1$ and $C_2$ such that, for every arc $I\subset \T$, $$\frac{1}{|I|}\int_I \log d_E^{-1}(\zeta) \; dm(\zeta) \leq C_1 \log\frac{1}{|I|}+C_2.$$ Their proof involves formulating an equivalent dual problem expressed in terms of distributions. Bruna \cite{brunaint} gave a constructive proof of this result using some of the ideas from his solution of the interpolation problem for finitely many derivatives. He proved that porosity is equivalent to having the Alexander-Taylor-Williams condition hold with $C_1=1$. Thus, every  porous set is a boundary interpolation set for $A^\infty$, but (as Bruna shows by giving an example) the converse is false.

Dyn{\textsc{\char13}}kin and Hru\v{s}\v{c}\"{e}v \cite{dynhru} studied interpolation for certain subclasses of $A^\infty$, including the  Gevrey class $G_\alpha$ (for $0<\alpha < 1$) defined in Section \ref{section:lip}.
They characterized the boundary interpolation sets $E$ for $G_\alpha$ as those satisfying the following condition: For all arcs $I\subset \T$,
\begin{equation} \label{eq:kalpha} \int_I d_E^{-\alpha}(\zeta) \; dm(\zeta) \leq C |I|^{1-\alpha}. \end{equation}
Following Dyn{\textsc{\char13}}kin \cite{dynkin}*{page 103}, we denote this condition by (K$_\alpha$).
Dyn{\textsc{\char13}}kin noted  that, for every $\alpha$, (K$_\alpha$) implies porosity, and he proved that if $E$ is porous then for some $\alpha$ it satisfies (K$_\alpha$). Thus, $E$ is a boundary  interpolation set for some (or all) $A^\alpha$ if and only if it is a boundary  interpolation set for some $G_\alpha$. It is unclear why the two problems are so closely related.

We mention interpolation results for some other function spaces. Dyn{\textsc{\char13}}kin's original paper \cite{dynkin} proved that  porosity characterizes boundary interpolation sets for elements of $A$ having derivatives in a Hardy space, the setting in which Koto\v{c}igov introduced this condition. In relation to Lipschitz classes, more recent results  (most of which are stated in terms of a modulus of continuity) are  Koto\v{c}igov \cite{kotocigov2}, Vasin \cites{vasin,vasin2}, and Mac\'{\i}a-Tugores \cite{mactug}.

For results about boundary interpolation sets for Lipschitz classes in $\Cn$, see Bruna-Ortega \cite{bruort}, Chaumat-Chollet \cite{chacho}, Saerens \cite{saer}, and Verdoucq \cite{verd}.

\vspace{.2in}

\section{The disc algebra $A$}\label{section:discalg}

\vspace{.2in}

In the case of the disc algebra $A$ both the peak sets and the boundary interpolation sets have a simple characterization: They are the closed subsets of $\T$ with Lebesgue measure zero. In fact, techniques from the theory of uniform algebras (closed subalgebras of the space of continuous functions on a compact Hausdorff space) give interesting alternative characterizations of these sets. In this section we state these results. Because they are well known, we will be fairly brief.

The case of peak sets was settled over a hundred years ago by Fatou \cite{fatou}*{page 393}, who constructed what we are calling a strong support function in $A$ for each closed subset $E$ of $\T$ with measure zero. Briefly, one begins with an integrable function $\psi$ that is smooth on $\T\setminus E$, continuous on $\T$ as a map to $[1,\infty]$, and equal to $\infty$ on $E$. If $u$ is the Poisson integral of $\psi$ and $v$ is the conjugate function, then $1/(u+iv)$ extends to an element of $A$ that is a strong support function for $E$.

Rudin \cite{rudin} and Carleson \cite{carlz} proved the following result on interpolation for the disc algebra.

\begin{thm}[Rudin-Carleson Theorem]\label{thm:r-c}
 Every closed subset of $\T$ with measure zero is an interpolation set for $A$.
\end{thm}

\begin{proof}[Sketch of proof] First we consider the case when the given set $E$ is the disjoint union of the closed sets $E_0$ and $E_1$ and the function $\phi$ to be interpolated is the simple function that is $0$ on $E_0$ and $1$ on $E_1$. For each $j$ let $g_j\in A$ be a strong support function for $E_j$. By taking square roots we can require that for each $j$ we have $|\mbox{arg }g_j|<\pi/4$ on $\overline{\D}\setminus E_j$. Define $g\in A$ by $g=g_0/(g_0+g_1)$. Then $g$ interpolates $\phi$. Note the bounds $0\leq \mbox{Re }g\leq 1$.

In the general case, because $E$ is totally disconnected, each continuous function on $E$ can be approximated uniformly by continuous simple functions on $E$. Then interpolation on $E$ can be accomplished by adapting the special case described in the preceding paragraph, where by means of a conformal map we make use of the bounds to ensure convergence.
\end{proof}

\vspace{.1in}

Danielyan \cite{daniel} gives an elementary proof that Fatou's result implies the Rudin-Carleson theorem.

Now we turn to alternative characterizations of the closed subsets of $\T$ with Lebesgue measure zero.
Rudin's proof, sketched above, shows in fact that every closed subset $E$ of $\T$ with measure zero is a peak-interpolation set (as in Section \ref{section:prelim}) for $A$: Given a continuous function $\phi$ on $E$ that is not identically zero, there exists $f\in A$ such that $f(z)=\phi(z)$ when $z\in E$ and $|f(z)|<\max\{|\phi(w)|\colon w\in E\}$ when
$z\in \overline{\D}\setminus E$. Of course, in general every peak-interpolation set is both a peak set and an interpolation set.

The results above on peak sets and interpolation sets are closely related to the theorem of F. and M. Riesz \cite{fmriesz} on absolute continuity of certain measures. That classical theorem can be reformulated as follows: Let $\mu$ be a complex Borel measure on $\T$ that annihilates $A$, in the sense that $\int f \; d\mu=0$ for every $f\in A$. Then each closed subset of $\T$ with Lebesgue measure zero is a null set for $\mu$. This theorem follows from Fatou's result on peak sets: Let $\mu$ be a complex Borel measure on $\T$ that annihilates $A$, and let $F$ be a closed subset of $\T$ with Lebesgue measure zero. Given a closed subset $E$ of $F$, find a peak function $f\in A$ for $E$. Then the sequence $\{f^n\}$ tends pointwise on $\T$ to the characteristic function of $E$, and $|f^n|\leq 1$ for all $n$. Thus, by the dominated convergence theorem $$\mu(E)=\lim\int f^n \; d\mu=0,$$ where the last equality holds since $\mu$ annihilates $A$. Hence, $F$ is a null set for $\mu$, as desired. This is essentially the proof in \cite{fmriesz}.

In the reverse direction, Bishop \cite{bishop}
(see also Glicksberg \cite{glick})
proved that the Rudin-Carleson theorem follows from the theorem of F. and M. Riesz. In fact, he considered the general setting of peak interpolation in closed subspaces of the space of continuous functions on a compact Hausdorff space. Since these results were proved, techniques from uniform algebras have been used to prove that for a variety of domains $\Omega$ in $\Cn$ the properties we have considered are equivalent for the space $A(\Omega)$ of functions holomorphic  on $\Omega$ and continuous up to the boundary. For example, if $E$ is a closed subset of the boundary of a bounded strictly pseudoconvex domain with smooth boundary (or a closed subset of the torus, the distinguished boundary of the polydisc), then the following five properties of $E$ are equivalent: zero set, peak set, interpolation set, peak-interpolation set, null set for every complex Borel measure on the boundary (or distinguished boundary) that annihilates $A(\Omega)$. For further details and references, see Rudin's books \cite{rupoly}*{Chapter 6} and \cite{ruball}*{Chapter 10}.

 Pe{\l}czy\'{n}ski \cite{pelc} used this work of Bishop to prove that boundary interpolation for $\T$ (and in more general settings) can be done using a linear isometry from the Banach space of continuous functions on $E\subset \T$ to $A$. See also Michael-Pe{\l}czy\'{n}ski \cite{mipel}.

 \vspace{.2in}

\section{Comparison of peak sets and boundary interpolation sets}\label{section:misc}

\vspace{.2in}

We have seen that, among closed subsets of $\T$, the zero sets, peak sets, interpolation sets, and peak-interpolation sets for the disc algebra coincide: They are the sets having Lebesgue measure zero.
This result fails for the other function spaces we have considered. For example, the peak sets for $A^\infty$ are finite, but the boundary interpolation sets are characterized by the Alexander-Taylor-Williams condition.

The comparison of peak sets and boundary interpolation sets for $A^\alpha$, $0<\alpha<1$, is particularly interesting. The characterization of boundary interpolation sets as porous is independent of $\alpha$, but the collection of peak sets for $A^\alpha$ depends on $\alpha$.
Bruna \cite{brunapk}*{pages 270--271} illustrated the situation with two examples: If $E=\{1\}\cup\{\exp(i/n)\colon n=1, 2, \ldots \}$, then $E$ is a peak set for $A^\alpha$ when $\alpha < 1/2$; but $E$ is not porous, so  it is an interpolation set for no $A^\alpha$. The middle-thirds Cantor set (considered as a subset of $\T$) is porous, so it is an interpolation set for each $A^\alpha$; but it is a not a peak set for $A^\alpha$ when $\alpha \geq 1-\log 2/\log 3$.

Bruna \cite{brunapk} showed that every boundary interpolation set for $A^\beta$ is a peak set for some $A^\alpha$. This is a consequence of his result that every set satisfying the condition (K$_\alpha$) defined by inequality \eqref{eq:kalpha} is a peak set for $A^\alpha$. (In fact, he shows that $E$ satisfies (K$_\alpha$)  if and only if $E$ is a strong peak set for $A^\alpha$, in the sense that $E$ has a strong support function $g\in A^\alpha$ such that $\mbox{Re }g \geq Cd_E^\alpha$ on $\overline{\D}$.) As we noted earlier, every porous set satisfies (K$_\alpha$) for some $\alpha$. Thus, every porous set is a peak set for some $A^\alpha$.
It is an open problem to characterize those porous sets that are peak sets for a given $A^\alpha$.

\def\MR#1{\relax\ifhmode\unskip\spacefactor3000 \space\fi%
  \href{http://www.ams.org/mathscinet-getitem?mr=#1}{MR#1}}

%amsrefs

\end{document}